\documentclass[10pt]{amsart}
\usepackage{amsmath, amsthm, amssymb, amsfonts, verbatim}

\usepackage{soul} 
\usepackage{mathabx}

\usepackage{array}
\usepackage{multirow}
\usepackage{hyperref}
\usepackage{setspace}

\usepackage{tikz}
\usetikzlibrary{matrix,arrows,decorations.pathmorphing}
\usepackage{newfloat}

\numberwithin{equation}{section}

\newcommand{\ra}{\rightarrow}

\newcommand{\e}{\epsilon}
\newcommand{\R}{{\mathbb R}}
\newcommand{\X}{{\mathbb X}}

\newcommand{\pd}{\partial}
\newcommand{\grad}{\operatorname{grad}}
\newcommand{\curl}{\operatorname{curl}}
\newcommand{\rot}{\operatorname{rot}}
\renewcommand{\div}{\operatorname{div}}

\newcommand{\sym}{\operatorname{sym}}
\newcommand{\skw}{\operatorname{skw}}
\newcommand{\tr}{\operatorname{tr}}

\newtheorem{theorem}{Theorem}[section]

\newtheorem{lemma}[theorem]{Lemma}
\newtheorem{corollary}[theorem]{Corollary}
\newtheorem{definition}[theorem]{Definition}

\def\be{\begin{equation*}}
\def\ee{\end{equation*}}

\begin{document}

\title[Unified analysis of mixed methods for elasticity]{Towards a unified analysis of mixed methods for elasticity with weakly symmetric stress}
\author{Jeonghun J. Lee
}


\address{University of Oslo
Department of Mathematics \\
P.O. Box 1053, Blindern 0316, Norway \\
jeonghul@math.uio.no}

\keywords{linear elasticity, weakly symmetric stress, mixed finite elements, error analysis}
\subjclass{65N30, 65N12}

\maketitle

\begin{abstract}
We propose a framework for unified analysis of mixed methods for elasticity with weakly symmetric stress. Based on a commuting diagram in the weakly symmetric elasticity complex and extending a previous stability result, stable mixed methods are obtained by combining Stokes stable and elasticity stable finite elements. We show that the framework can be used to analyze most existing mixed methods for the elasticity problem with elementary techniques. We also show that some new stable mixed finite elements are obtained.
\end{abstract}


\section{Introduction}

In the Hellinger--Reissner formulation of linear elasticity, for a given external body force and boundary conditions, the stress and displacement are sought as a saddle point of the Hellinger--Reissner functional. In this saddle point problem the stress tensor is directly obtained without suffering from volumetric locking in nearly incompressible materials \cite{ADG84}. However, the symmetry condition of the stress tensor gives a highly nontrivial obstacle in finding stable mixed finite elements for the saddle point problem. Another way to find mixed methods for the problem is to impose the symmetry condition weakly by requiring the stress to be orthogonal against a certain space of skew-symmetric tensors \cite{AmaraThomas,Fraeijs}. This alternative approach turned out to be successful and various stable mixed finite elements have been developed based on this idea \cite{ABD84,AFW07,MR2449101,CGG10,FalkWS,MR1464150,GG10,morley89,Sten86,Sten88}. In this paper we will call them weak symmetry elements. 

There are several different ways to analyze the stability of weak symmetry elements. In early research \cite{ABD84,MR1464150} a connection between the Stokes equation and the linear elasticity equation with weak symmetry is used for the proof of stability. An analysis using mesh-dependent norms is also proposed \cite{Sten86,Sten88}. A breakthrough is made in \cite{AFW07} in the development of a weakly symmetric elasticity complex. In \cite{AFW07} the Arnold--Falk--Winther (AFW) family is developed and the stability is proved by commuting diagram properties in the elasticity complex and a diagram chasing type argument in homological algebra. In \cite{MR2449101} an analysis, based on a connection with the Stokes equation, is revived with a commuting diagram in the weakly symmetric elasticity complex \cite{AFW07}, and the stability proof of the AFW family is reduced to proving existence of an interpolation operator satisfying some conditions. As a consequence, new elements are developed and an alternative stability proof is proposed in \cite{MR2449101}. This idea is adopted in \cite{CGG10,GG10} to construct new elements with the aid of cleverly-designed matrix bubble functions and some results in \cite{AFW07}. 

The aforementioned ways for the stability proof, although they are interesting, are not elementary, and some of them need sophisticated concepts which are not familiar to many numerical analysts and engineers. The goal of this paper is to provide a unified framework for the analysis of weak symmetry elements with {\it elementary techniques}. It is worth to mention that there is a similar attempt with mesh-dependent norms \cite{Sten14}, which can be useful for developing mixed discontinuous Galerkin methods for the problem. In contrast, we revisit the approach in \cite{MR2449101} and extend it using an idea in \cite{Guzman11}. We establish an abstract framework with this simple extension and show that various known results can be easily recovered. For example, in our approach, the stability of variable degree finite elements on affine meshes \cite{MR2557491,MR2792388} is proved without special interpolation operators. 

Another contribution of our work is to prove an improved error estimate, which was available only for several families of elements \cite{AFW07,CGG10,GG10,Guzman11,Sten88}, for a wider class of elements. We prove that the improved error estimate holds for {\it all} elements satisfying several simple conditions and we shall show through examples that most known finite elements, including variable degree elements, satisfy the conditions. It is worth mentioning that this improved error estimate leads to existence of weakly symmetric elliptic projection, which is a key tool for error analysis of time-dependent problems \cite{Arnold-Lee}.

This paper is organized as follows. In section 2, we summarize notation and review the Hellinger--Reissner formulation of linear elasticity with weakly symmetric stress. In section 3, we introduce an abstract framework for unified analysis and prove an improved a priori error estimates. In section 4, we give examples of weak symmetry elements to which the abstract framework can be applied. 

\section{Preliminaries}
\subsection{Notation} Let $\Omega$ be a bounded domain in $\R^n$ ($n=2,3$) with a Lipschitz boundary. We use $H^m (\Omega)$, $m \geq 0$ to denote standard Sobolev spaces \cite{Evans-book} based on the $L^2$ norm ($H^0(\Omega) = L^2(\Omega)$) and for a finite dimensional inner product space $\X$, $H^m(\X)$ is the space of $\X$--valued functions such that each component is in $H^m(\Omega)$. The associated norm is denoted with $\| \cdot \|_m$. For $p, q \in L^2(\X)$ we will use $(p , q)$ to denote the $L^2$ inner product. We denote the spaces of all, symmetric, and skew-symmetric $n \times n$ matrices by $\R^{n \times n}$, $\R_{\sym}^{n \times n}$, and $\R_{\skw}^{n \times n}$, respectively.

We use $\grad$ and $\div$ to denote the standard gradient and divergence operators. However, we use $\curl$ to denote two different operators for different $n$, namely, if $n=2$, then 
\begin{align*}
\curl : H^1(\R) \ra L^2(\R^2), \qquad \curl \phi = (-\pd_{x_2} \phi \quad \pd_{x_1} \phi),
\end{align*}
and if $n=3$, then $\curl$ is the standard three dimensional curl operator. 

By $H(\div)$ we denote the space of square integrable $\R^n$-valued functions on $\Omega$ such that the divergence of functions is also square integrable and the $H(\div)$ norm is defined by $\| \tau \|_{\div}^2 = \| \tau \|_0^2 + \| \div \tau \|_0^2$. The $H(\curl)$ space and $\| \cdot \|_{\curl}$ are defined similarly if $n=3$. When we apply the operators $\grad$, $\div$, $\curl$ to a matrix-valued or vector-valued function the operations need to be well-defined as row-wise operators. By $H(\div; \R^n)$ we denote the space of functions in $L^2(\R^{n \times n})$ such that each row is in $H(\div)$. The space $H(\curl;\R^3)$ is defined similarly for $n=3$. The $H(\div)$ and $H(\curl)$ norms of the spaces are naturally defined.

\subsection{Hellinger--Reissner formulation of linear elasticity}

For a given displacement $u : \Omega \rightarrow \R^n$, the linear strain tensor $\epsilon(u)$ is 
\begin{align*}
\e(u) = \frac{1}{2} ( \grad u + (\grad u)^T),
\end{align*} 
where $(\grad u)^T$ is the transpose of $\grad u$. 
From generalized Hooke's law the stress tensor is $\sigma = C \epsilon(u)$ where $C$ is the stiffness tensor such that $C(x) : \R_{\sym}^{n \times n} \ra \R_{\sym}^{n \times n}$ for all $x \in \Omega$ and 
\begin{align*}
c_0 \tau : \tau \leq C(x) \tau : \tau \leq c_1 \tau : \tau ,\quad \tau \in \R_{\sym}^{n \times n},
\end{align*}
with positive constants $c_0, c_1$ independent of $x \in \Omega$. For each $x \in \Omega$, $C(x)^{-1}$ is also bounded and positive definite. If an elastic medium is isotropic, then $C^{-1}\tau$ has the form
\begin{align}
\label{compliance} C^{-1} \tau = \frac{1}{2\mu} \left( \tau - \frac{\lambda}{2\mu + n\lambda} \tr(\tau) I \right),
\end{align}
where $\mu(x), \lambda(x) >0$ are the Lam\'{e} parameters, $\tr(\tau)$ is the trace of $\tau$, and $I$ is the identity matrix. 

Throughout this paper we assume the homogeneous displacement boundary condition $u=0$ on $\pd \Omega$ for simplicity. For a given $f \in L^2(\R^n)$, the Hellinger--Reissner functional $\mathcal{J} : (H(\div ; \R^n) \cap L^2(\R_{\sym}^{n \times n})) \times L^2(\R^n) \ra \R$ is defined by 
\begin{align}
\label{J}
\mathcal{J}(\tau, v) = \int_{\Omega} \left( \frac{1}{2} C^{-1} \tau : \tau + \div \tau \cdot v - f \cdot v \right ) dx,
\end{align}
and it is known that $\mathcal{J}$ has a unique critical point 
\begin{align*}
(\sigma, u) \in H(\Omega, \div ; \R_{\sym}^{n \times n}) \times L^2(\Omega; \R^n),
\end{align*} 
which is the solution of the elasticity problem with the boundary condition $u = 0$. 

For the approach with weakly imposed symmetry of stress we define $A$ as the extension of $C^{-1}$ on $\R^{n \times n}$ such that $A$ is the identity map for skew-symmetric matrices. We define function spaces $\Sigma$, $U$, and $\Gamma$ by
\begin{gather*}
\Sigma =  H(\div; \R^n), \quad U = L^2(\R^n), \quad \Gamma = L^2(\R_{\skw}^{n \times n}),
\end{gather*}
and a functional $\tilde {\mathcal{J}} : \Sigma \times U \times \Gamma \ra \R$ by
\begin{gather}
\label{J-tilde}
\tilde {\mathcal{J}} (\tau, v, \eta) = \int_{\Omega} \left( \frac{1}{2} A \tau : \tau + \div \tau \cdot v + \tau : \eta  - f \cdot v \right ) dx.
\end{gather}
The functional $\tilde {\mathcal{J}}$ has a unique critical point $(\sigma, u, \gamma)$ (see \cite{AFW06}) and the first two components coincide with the critical point of $\mathcal{J}$ in \eqref{J}. By variational methods, the critical point $(\sigma, u, \gamma)$ of $\tilde {\mathcal{J}}$ satisfies 
\begin{align}
\label{eq:weak-eq1} (A \sigma, \tau) + (u, \div \tau) + (\gamma, \tau) &= 0, & & \tau \in \Sigma, \\
\label{eq:weak-eq2} -(\div \sigma, v) &= (f, v), & & v \in U, \\
\label{eq:weak-eq3} (\sigma, \eta) &= 0, & & \eta \in \Gamma. 
\end{align}
The associated discrete problem with finite element spaces $\Sigma_h \times U_h \times \Gamma_h \subset \Sigma \times U \times \Gamma$ is seeking $(\sigma_h, u_h, \gamma_h) \in \Sigma_h \times U_h \times \Gamma_h$ such that 
\begin{align}
\label{eq:weak-disc-eq1} (A \sigma_h, \tau) + (u_h, \div \tau) + (\gamma_h, \tau) &= 0, & & \tau \in \Sigma_h, \\
\label{eq:weak-disc-eq2} -(\div \sigma_h, v) &= (f, v), & & v \in U_h, \\
\label{eq:weak-disc-eq3} (\sigma_h, \eta) &= 0, & & \eta \in \Gamma_h. 
\end{align}
In this approach the numerical stress $\sigma_h$ is not symmetric but is weakly symmetric due to the last equation of the above.

\section{Abstract framework}
In this section we introduce an abstract framework for unified analysis of weak symmetry elements. This is a generalization of the approach in \cite{JL1} with inspirations from \cite{MR2449101,Guzman11}. 

Throughout this paper $c$ is a generic positive constant independent of mesh sizes. We first recall the Babu\v{s}ka--Brezzi stability conditions for \eqref{eq:weak-disc-eq1}--\eqref{eq:weak-disc-eq3}, which are
\begin{itemize}
\item[\bf(S1)] There is $c$ such that 
\begin{align*}
c\| \tau \|_{\div}^2 \leq (A \tau, \tau), 
\end{align*}
for $\tau \in \Sigma_h$ satisfying $(\div \tau, v) + (\tau, \eta) = 0$ for all $(v, \eta) \in U_h \times \Gamma_h$.
\item[\bf (S2)] There is $c$ such that 
\begin{align*}
\inf_{0 \not = (v, \eta) \in U_h \times \Gamma_h} \sup_{0 \not = \tau \in \Sigma_h} \frac{(\div \tau, v) + (\tau, \eta)}{\| \tau \|_{\div} (\| v \|_0 + \| \eta \|_0)} \geq c . 
\end{align*}
\end{itemize}
Now we recall a commuting diagram of the elasticity complex in \cite{AFW07}. Let 
\begin{align*}
\Xi = 
\begin{cases}
H^1(\R^2)  & \text{if }n=2, \\
H(\curl;\R^3) &\text{if }n=3.
\end{cases}
\end{align*}
We also define $S$ and $\chi$ as
\begin{align*}
S 
\begin{pmatrix}
\xi_1 \\
\xi_2
\end{pmatrix} = \frac{1}{2}
\begin{pmatrix}
\xi_1 & \xi_2
\end{pmatrix} \quad \text{for } \xi \in \Xi, \quad \chi(r) = 
\begin{pmatrix}
0 & r \\
-r & 0
\end{pmatrix} \quad \text{for } r \in \R \qquad 
\text{ if } n = 2, \\
S \xi = \frac{1}{2}(\xi^T - (\tr \xi) I) \quad \text{for } \xi \in \Xi, \quad \chi 
\begin{pmatrix}
r_1 \\
r_2 \\
r_3
\end{pmatrix} = 
\begin{pmatrix}
0 & -r_3 & r_2 \\
r_3 & 0 & -r_1 \\
-r_2 & r_1 & 0
\end{pmatrix} \quad \text{if } n = 3. 
\end{align*} 
Note that $S$ and $\chi$ are invertible algebraic operators. One can verify by a direct computation that $S$ maps $\Xi$ to $H(\div)$ if $n=2$, and to $H(\div; \R^3)$ if $n=3$, so $\chi \div S$ maps $\Xi$ to $\Gamma$. One can also verify by a direct computation that 
\begin{align*}
\skw \curl \xi = \chi \div S \xi, \qquad \xi \in \Xi, 
\end{align*}
where $\skw \tau = (\tau - \tau^T)/2$ for $\tau \in L^2(\Omega; \R^{n \times n})$. For finite element spaces $\Xi_h \subset \Xi$, $\Gamma_h \subset \Gamma$, and the $L^2$ projection $Q_h$ into $\Gamma_h$, it holds that
\begin{align*}
Q_h \skw \curl \xi = Q_h \chi \div S \xi, \qquad \xi \in \Xi_h,
\end{align*}
which implies that the triangle in Figure \ref{diag1} commutes.
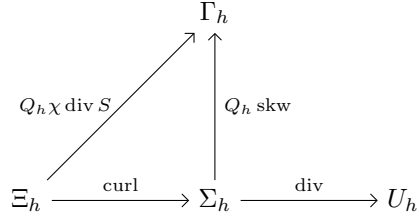
\begin{figure}[!ht] 
\centering
\begin{tikzpicture}[scale=2.5] 
\node (L2-skw) 	at (1,1) {$\Gamma_h$};
\node (Xi)		at (0,0) {$\Xi_h$};
\node (Hdiv-vec) 	at (1,0) {$\Sigma_h$};
\node (L2-vec)  		at (2,0) {$U_h$}; 
\path[->, font=\scriptsize, >=angle 90]
(Xi)		edge node[above] {$\curl$} (Hdiv-vec)
(Hdiv-vec) 	edge node[above] {$\div$} (L2-vec)
(Xi)		edge node[left] {$Q_h \chi \div S$}	(L2-skw)
(Hdiv-vec)	edge node[right] {$Q_h \skw$}	(L2-skw);
\end{tikzpicture}
\caption{A finite element version of the commuting diagram in the weakly symmetric elasticity complex}
\label{diag1}
\end{figure}
Note that the bottom row of this diagram is not necessarily an exact sequence.

\begin{definition}
A triple of finite elements $(\Sigma_h, U_h, R_h)$ is elasticity stable if $\Sigma_h \subset \Sigma$, $U_h \subset U$, $R_h \subset \Gamma$ and the following hold:
\begin{itemize}
\item[{\bf (A1)}] $\div \Sigma_h = U_h$
\item[{\bf (A2)}] There exists $c$ such that for any $(v, \rho) \in U_h \times R_h$, there exists $\tau \in \Sigma_h$ satisfying
\begin{align*}
\div \tau = v, \quad (\tau, \rho') = ( \rho, \rho') \;\;\forall \rho' \in R_h, \quad \| \tau \|_{\div} \leq c(\| v \|_0 + \| \rho \|_0).
\end{align*}

\end{itemize}
\end{definition}
It is not difficult to see that {\bf (A1)} implies {\bf (S1)} because $A$ is positive definite, and {\bf (A2)} implies {\bf (S2)}. However, an elasticity stable triple $(\Sigma_h, U_h, R_h)$ may not be an appropriate mixed finite element for \eqref{eq:weak-disc-eq1}--\eqref{eq:weak-disc-eq3} in the standard context. For instance, if $\Sigma_h =BDM_1(\R^2)$, $U_h = \mathcal{P}_0^d(\R^2)$, then $(\Sigma_h, U_h, 0)$ is elasticity stable. Thus an elasticity stable triple $(\Sigma_h, U_h, \Gamma_h)$ should have a reasonable order of approximation of $(\Sigma, U, \Gamma)$ to be an appropriate mixed finite element for \eqref{eq:weak-disc-eq1}--\eqref{eq:weak-disc-eq3}.

\begin{definition} Suppose that $\Xi_h \subset \Xi$, $R_h \subset \Gamma$ are finite element spaces and $Q_h : \Gamma \ra R_h$ is the $L^2$ projection. The pair $(\Xi_h, R_h)$ is Stokes stable if 

\begin{itemize}
\item[\bf (B)] There exists $c$ such that for any $\rho \in R_h$ there is $\xi \in \Xi_h$ satisfying 
\begin{align*}
(Q_h \chi \div S \xi, \rho) \geq c \| \rho \|_0^2, \qquad \| \curl \xi \|_0 \leq c \| \rho \|_0.
\end{align*}
\end{itemize}
\end{definition}
The condition {\bf (B)} implies that $Q_h \chi \div S : \Xi_h \ra R_h$ is surjective. Furthermore, for any $\rho \in R_h$, there exists $\xi \in \Xi_h$ such that $Q_h \chi \div S \xi = \rho$ and $\| \curl \xi \|_0 \leq c \| \rho \|_0$.

\begin{theorem} Let $\Xi_h \subset \Xi$, $\Sigma_h \subset \Sigma$, $U_h \subset U$, $\Gamma_h \subset \Gamma$ be four finite element spaces. For a subspace $\Gamma_h^0$ of $\Gamma_h$ its orthogonal complement is denoted by $\Gamma_h^1$. Suppose that $(\Sigma_h, U_h, \Gamma_h^0)$ is elasticity stable and $(\Xi_h, \Gamma_h^1)$ is Stokes stable with $\chi \div S \Xi_h \perp \Gamma_h^0$. Then $(\Sigma_h, U_h, \Gamma_h)$ is elasticity stable. 
\end{theorem}
\begin{proof}
Since we assume that $(\Sigma_h, U_h, \Gamma_h^0)$ is elasticity stable, we only need to check {\rm \bf{(A2)}} to show that $(\Sigma_h, U_h, \Gamma_h)$ is elasticity stable. Let $v \in U_h$, $\eta = \eta_0 + \eta_1 \in \Gamma_h^0 \oplus \Gamma_h^1$ be given. By {\bf (A2)} for $(\Sigma_h, U_h, \Gamma_h^0)$ there is $\tau_0 \in \Sigma_h$ such that 
\begin{gather*}
\| \tau_0 \|_{\div} \leq c (\| v \|_0 + \| \eta \|_0), \qquad 
\div \tau_0 = v, \qquad (\tau_0, \eta_0') = (\eta_0, \eta_0'), \quad \forall \eta_0' \in \Gamma_h^0.
\end{gather*}
Let $Q_h^1$ be the $L^2$ projection into $\Gamma_h^1$. Since $(\Xi_h, \Gamma_h^1)$ is Stokes stable, there exists $\xi \in \Xi_h$ such that $Q_h^1 \chi \div S \xi = \eta_1 - Q_h^1 \tau_0$ and $\| \curl \xi \|_0 \leq c \| \eta_1 - Q_h^1 \tau_0 \|_0$. We take $\tau = \tau_0 + \curl \xi$ and check that the conditions in {\bf (A2)} holds. Since $\div \tau_0 = v$,   
\begin{align*}
\div \tau = \div \tau_0 + \div \curl \xi = \div \tau_0 = v.
\end{align*}
For $\eta' = \eta_0' + \eta_1' \in \Gamma_h^0 \oplus \Gamma_h^1$ 
\begin{align*}
(\tau, \eta') &= (\tau_0 + \skw \curl \xi, \eta') \\
&= ( \tau_0 + \chi \div S \xi , \eta') \\
&= ( \tau_0, \eta_0') +  ( \tau_0 + \chi \div S \xi , \eta_1') \qquad (\because \chi \div S \xi \perp \eta_0') \\
&= (\tau_0, \eta_0') + (Q_h^1 (\tau_0 + \chi \div S \xi ), \eta_1') \\
&= (\eta_0, \eta_0') + (\eta_1, \eta_1') \qquad (\because Q_h^1 \chi \div S \xi = \eta_1 - Q_h^1 \tau_0)\\
&= (\eta, \eta').
\end{align*}
Note that $\| \eta_1 \|_0 \leq \| \eta \|_0$. By the triangle inequality, $\| \tau_0 \|_{\div} \leq c(\| v \|_0 + \| \eta \|_0)$ and $\| \curl \xi \|_0 \leq c \| \eta_1 - Q_h^1 \tau_0 \|_0 \leq c (\| \eta \|_0 + \| \tau_0 \|_0)$, 
\begin{align*}
\| \tau \|_{\div} \leq \| \curl \xi \|_0 + \| \tau_0 \|_{\div} \leq c (\| v \|_0 + \| \eta \|_0).
\end{align*}
Thus {\bf (A2)} holds and $(\Sigma_h, U_h, \Gamma_h)$ is elasticity stable.
\end{proof}
Now we show an improved a priori error estimate. 
\begin{theorem} \label{thm:error-estm} Suppose that $(\Sigma_h, U_h, \Gamma_h)$ is elasticity stable and there exists an interpolation operator $\Pi_h : H^1(\Omega, \R^{n \times n}) \ra \Sigma_h$ such that 
\begin{align*}
\div \Pi_h \tau = P_h \div \tau, \quad \tau \in H^1(\Omega; \R^{n \times n}),
\end{align*}
where $P_h$ is the $L^2$ projection into $U_h$. Let $(\sigma, u, \gamma)$ and $(\sigma_h, u_h, \gamma_h)$ be the solutions of \eqref{eq:weak-eq1}--\eqref{eq:weak-eq3} and \eqref{eq:weak-disc-eq1}--\eqref{eq:weak-disc-eq3}. Then the following inequality
\begin{align} \label{eq:imp-err}
\| \sigma - \sigma_h \|_0 + \| P_h u - u_h \|_0 + \| \gamma - \gamma_h \|_0 &\leq c(\| \sigma - \Pi_h \sigma \|_0 + \| \gamma - Q_h \gamma \|_0),
\end{align}
holds with $Q_h$ the $L^2$ projection into $\Gamma_h$.
\end{theorem}
\begin{proof}
The proof is same as that of Theorem 4.1 in \cite{JL1} but we include the details here to be self-contained. The difference of \eqref{eq:weak-eq1}--\eqref{eq:weak-eq3} and \eqref{eq:weak-disc-eq1}--\eqref{eq:weak-disc-eq3} gives  
\begin{align}
\label{eq:err-eq1} (A(\sigma - \sigma_h), \tau) + (u - u_h, \div \tau) + (\gamma - \gamma_h, \tau) &= 0, & & \tau \in \Sigma_h, \\
\label{eq:err-eq2} (\div (\sigma - \sigma_h), v) &= 0, & & v \in U_h, \\
\label{eq:err-eq3} (\sigma - \sigma_h, \eta) &= 0, & & \eta \in \Gamma_h.
\end{align}
Let $\Sigma_{h,0} = \{ \tau \in \Sigma_h : \div \tau = 0 \}$ and consider an auxiliary problem of seeking $(\sigma_h', \gamma_h') \in \Sigma_{h,0} \times \Gamma_h$ such that
\begin{align} \label{eq:reduced-system}
(A \sigma_h', \tau) + (\gamma_h', \tau) + (\sigma_h', \eta) = F(\tau) + G(\eta), \qquad (\tau, \eta) \in \Sigma_{h,0} \times \Gamma_h,
\end{align}
with a bounded linear functional $(F, G)$ on $\Sigma_{h,0} \times \Gamma_h$. As a special case of {\bf (A2)}, for $v=0$ and any given $\eta \in \Gamma_h$ there exists $\tau \in \Sigma_{h,0}$ such that $(\tau, \eta') = (\eta, \eta')$ for all $\eta' \in \Gamma_h$ and $\| \tau \|_0 \leq c \|\eta \|_0$.   From this observation and {\bf (A1)}, $\Sigma_{h,0} \times \Gamma_h$ is a stable mixed finite element for the problem \eqref{eq:reduced-system} with the $L^2$ norms. By restricting $\tau \in \Sigma_{h,0}$, the sum of \eqref{eq:err-eq1} and \eqref{eq:err-eq3} is
\begin{align*}
(A(\sigma - \sigma_h), \tau) + (\gamma - \gamma_h, \tau) + (\sigma - \sigma_h, \eta) = 0,
\end{align*}
which is equivalent to 
\begin{multline} \label{eq:inter-eq1}
(A(\sigma_h - \Pi_h \sigma), \tau) + (\gamma_h - Q_h \gamma, \tau) + (\sigma_h - \Pi_h \sigma, \eta) \\
= (A(\sigma - \Pi_h \sigma), \tau) + (\gamma - Q_h \gamma, \tau) + (\sigma - \Pi_h \sigma, \eta).
\end{multline}
Note that $\sigma_h - \Pi_h \sigma \in \Sigma_{h,0}$ because $\div \sigma_h = P_h \div \sigma = \div \Pi_h \sigma$ by \eqref{eq:err-eq2} and $\div \Sigma_h = U_h$. By the Babu\v{s}ka--Brezzi stability of \eqref{eq:reduced-system}, there exists $(\tau, \eta) \in \Sigma_{h,0} \times \Gamma_h$ such that $\| \tau \|_0 + \| \eta \|_0 \leq c$ and 
\begin{multline*}
\| \sigma_h - \Pi_h \sigma \|_0 + \| \gamma_h - Q_h \gamma \|_0 \leq (A(\sigma_h - \Pi_h \sigma), \tau) + (\gamma_h - Q_h \gamma, \tau) + (\sigma_h - \Pi_h \sigma, \eta).
\end{multline*}
Combining this, \eqref{eq:inter-eq1}, and the Cauchy--Schwarz inequality with $\| \tau \|_0 + \| \eta \|_0 \leq c$, 
\begin{align*}
\| \sigma_h - \Pi_h \sigma \|_0 + \| \gamma_h - Q_h \gamma \|_0 &\leq (A(\sigma - \Pi_h \sigma), \tau) + (\gamma - Q_h \gamma, \tau) + (\sigma - \Pi_h \sigma, \eta) \\
&\leq c (\| \sigma - \Pi_h \sigma \|_0 + \| \gamma - Q_h \gamma \|_0). 
\end{align*}
By the triangle inequality and the above one, 
\begin{align*}
\| \sigma - \sigma_h \|_0 + \| \gamma - \gamma_h \|_0 &\leq \| \sigma - \Pi_h \sigma \|_0 + \| \Pi_h \sigma - \sigma_h \|_0 + \| \gamma - Q_h \gamma \|_0 + \| Q_h \gamma - \gamma_h \|_0 \\
&\leq c (\| \sigma - \Pi_h \sigma \|_0 + \| \gamma - Q_h \gamma \|_0),
\end{align*}
so \eqref{eq:imp-err} for $\| \sigma - \sigma_h \|_0$ and $\| \gamma - \gamma_h \|_0$ is proved. To estimate $\| u_h - P_h u \|_0$, observe that \eqref{eq:err-eq1} gives
\begin{align} \label{eq:new-err-eq1}
(A(\sigma - \sigma_h), \tau) + (P_h u - u_h , \div \tau) + (\gamma - \gamma_h, \tau) = 0, \qquad \tau \in \Sigma_h,
\end{align}
because $\div \tau \in U_h$ is orthogonal to $u - P_h u$. By {\bf (A2)} there is $\tau$ in \eqref{eq:new-err-eq1} such that $\div \tau = P_h u - u_h$ and $\| \tau \|_{\div} \leq c \| P_h u - u_h \|_0$. Then we have
\begin{align*}
\| P_h u - u_h \|_0^2 &= -(A(\sigma - \sigma_h), \tau) - (\gamma - \gamma_h, \tau) \\
& \leq c (\| \sigma - \sigma_h \|_0 + \| \gamma - \gamma_h \|_0) \| P_h u - u_h \|_0.
\end{align*}
Combining the result with the estimates of $\| \sigma - \sigma_h \|_0$ and $\| \gamma - \gamma_h \|_0$, we have
\begin{align*}
\| P_h u - u_h \|_0 \leq c (\| \sigma - \Pi_h \sigma \|_0 + \| \gamma - Q_h \gamma \|_0) , \end{align*}
as desired. 
\end{proof}
If $\Sigma_h$ and $\Gamma_h$ provide higher order approximations than that of $U_h$, then \eqref{eq:imp-err} implies that $\| P_h u - u_h \|_0$ is superconvergent. A local post-processing can be used to get a new numerical solution $u_h^*$ such that the convergence rate of $\| u - u_h^* \|_0$ is as good as that of $\| \sigma - \sigma_h \|_0 + \| \gamma - \gamma_h \|_0$. A higher order superconvergence of $\| P_h u - u_h \|_0$ can be obtained by the Aubin--Nitsche duality argument and the elliptic regularity of $\Omega$ when $f \in U_h$. In this case, a higher order local post-processing can be used to obtain $u_h^{**}$ which is a higher order approximation of $u$ in $L^2(\R^n)$. A careful discussion can be found in \cite{JL1} for second order rectangular elements. It is straightforward to generalize the argument in \cite{JL1} to higher order elements.

\section{Examples}
In this section we show examples which can be analyzed by the abstract framework. 

By $\mathcal{T}_h$ we denote a shape-regular mesh of $\Omega$ and $h$ is the maximum diameter of the elements in $\mathcal{T}_h$. By $\mathcal{P}_k(D)$ and $\mathcal{P}_k(D;\mathbb{X})$, we denote the spaces of $\R$ and $\mathbb{X}$--valued polynomials of degree $\leq k$ on $D \subset \Omega$. For a rectangle $D$, $\mathcal{Q}_k(D)$ is the space of polynomials of degree at most $k$ in each variable $x_i$, $1 \leq i \leq n$. Now we define 
\begin{align*}
\mathcal{P}_k^d(\X) &= \{ p \in L^2(\X) \,|\, p|_T \in \mathcal{P}_k(T; \X), \quad T \in \mathcal{T}_h \}, & &  k \geq 0,\\
\mathcal{P}_k^c(\X) &= \{ p \in H^1(\X) \,|\, p|_T \in \mathcal{P}_k(T; \X), \quad T \in \mathcal{T}_h \}, & & k \geq 1,\\
\mathcal{Q}_k^d(\X) &= \{ p \in L^2(\X) \,|\, p|_T \in \mathcal{Q}_k(T; \X), \quad T \in \mathcal{T}_h \}, & &  k \geq 0,\\
\mathcal{Q}_k^c(\X) &= \{ p \in H^1(\X) \,|\, p|_T \in \mathcal{Q}_k(T; \X), \quad T \in \mathcal{T}_h \}, & & k \geq 1,\\
RTN_k &= \{ p \in H(\div) \,|\, p|_T \in \mathcal{P}_{k-1}(T; \R^n) + \boldsymbol{x} \mathcal{P}_{k-1}(T) \}, & & k \geq 1,\\
BDM_k &= \{ p \in H(\div) \,|\, p|_T \in \mathcal{P}_k(T; \R^n) \}, & & k \geq 1,
\end{align*}
where $\boldsymbol{x}$ is the vector function $(x_1, ..., x_n)$ \cite{BDM85,Nedelec80,Nedelec86,RT75}. Note that the lowest order RTN element is denoted by $RTN_1$ in this paper, which is different from \cite{BFBook}. The rectangular RTN and BDM elements \cite{BFBook} are denoted by $rRTN_k$ and $rBDM_k$ with $k\geq 1$.  We also define $RTN_k(\R^n)$ and $BDM_k(\R^n)$ as the subspaces of $H(\div;\R^n)$ such that each row of an element in those spaces is in $RTN_k$ and $BDM_k$, respectively. Throughout this section $\Sigma_h$, $U_h$, $\Gamma_h$, $\Xi_h$ are the finite element spaces in Figure \ref{diag1}.

\subsection{Elements with continuous $\Gamma_h$} 

\subsubsection{PEERS} \label{eg1}
Let $b_T$ be the standard cubic bubble function on a triangle $T \in \mathcal{T}_h$ and
\begin{align*}
B = \{ \xi \,|\, \xi|_T = p b_T , p \in \R^2, T \in \mathcal{T}_h \}.
\end{align*} 
The PEERS \cite{ABD84} is 
\begin{align*}
\Sigma_h = RTN_1(\R^2) + \curl B, \qquad U_h = \mathcal{P}_0^d(\R^2), \qquad \Gamma_h = \mathcal{P}_1^c(\R_{\skw}^{2 \times 2}).
\end{align*}
Let $\Gamma_h^0 = 0$ and $\Xi_h = \mathcal{P}_1^c(\R^2) + B$. It is not difficult to see that $(\Sigma_h, U_h, 0)$ is elasticity stable from the stability of $(RTN_1, \mathcal{P}_0^d)$ for the mixed Poisson equation. Moreover, the stability of the MINI element for the Stokes equation \cite{ABF84} implies that $(\Xi_h, \Gamma_h)$ is Stokes stable because $S$ and $\chi$ are algebraic isomorphisms. Thus $(\Sigma_h, U_h, \Gamma_h)$ is elasticity stable. 

\subsubsection{Taylor--Hood based elements} \label{eg2} Let
\begin{align*}
\Sigma_h = BDM_{k}(\R^n), \quad U_h = \mathcal{P}_{k-1}^d (\R^n), \quad \Gamma_h = \mathcal{P}_{k}^c(\R_{\skw}^{n \times n}),
\end{align*}
for $k \geq 1$ and take $\Gamma_h^0 = 0$, 
\begin{align*}
\Xi_h = 
\begin{cases}
\mathcal{P}_{k+1}^c (\R^2), &\text{ if } n = 2, \\
\mathcal{P}_{k+1}^c (\R^{3 \times 3}), &\text{ if } n = 3.
\end{cases}
\end{align*}
By definition, $(\Sigma_h, U_h, 0)$ is elasticity stable. Moreover, the stability of Taylor--Hood elements for the Stokes equation \cite{MR1269482,MR1442933,MR1098408,MR0339677} yields that $(\Xi_h, \Gamma_h)$ is Stokes stable. Thus $(\Sigma_h, U_h, \Gamma_h)$ is elasticity stable. 

In the two dimensional case these elements were noticed in \cite{FalkWS}. In the three dimensional case it seems that the same elements have not appeared in the literature but similar elements were proposed in \cite{MR2449101} with slightly larger space for $\Sigma_h$ using the Raviart--Thomas--N\'{e}d\'{e}lec spaces. In addition, the improved error estimate \eqref{eq:imp-err} was not claimed in \cite{MR2449101}.

\subsubsection{Two dimensional rectangular element} \label{eg3} Let $\Omega$ be a bounded two dimensional domain with a rectangular mesh $\mathcal{T}_h$. There is a rectangular PEERS element which has same convergence rates as PEERS \cite{morley89}. Here we propose a rectangular version of the Taylor--Hood based elements in two dimensions. Let $S_2$ be the serendipity element with 8 local degrees of freedom \cite{Brenner-Scott-book} and set 
\begin{align*}
\Sigma_h = rBDM_1(\R^2), \quad U_h = \mathcal{P}_0^d(\R^2), \quad \Gamma_h = Q_1^c(\R_{\skw}^{2 \times 2}),
\end{align*}
and let $\Xi_h = S_2(\R^2)$, $\Gamma_h^0 = 0$. Since $(rBDM_1, \mathcal{P}_0^d)$ is stable for mixed Poisson equation with $\div rBDM_1 = \mathcal{P}_0^d$, $(\Sigma_h, U_h, 0)$ is elasticity stable. It is also known that $(\Xi_h, \Gamma_h)$ is Stokes stable \cite{MR725982}, so $(\Sigma_h, U_h, \Gamma_h)$ is elasticity stable. 

\subsection{Elements with discontinuous $\Gamma_h$}
To apply the framework for elements with discontinuous $\Gamma_h$ we need preliminary results.  
\begin{lemma} \label{lemma:afw1}
Suppose that $\Sigma_h = BDM_1(\R^n)$, $U_h = 0$, and $\Gamma_h = \mathcal{P}_0^d(\R_{\skw}^{n \times n})$. Then $(\Sigma_h, U_h, \Gamma_h)$ is elasticity stable.
\end{lemma}
A proof of this lemma can be found in \cite{AFW07,MR2449101,Guzman11}. For a simple proof we refer to Proposition 2.10 in \cite{Guzman11}. From this lemma the following can be easily obtained.
\begin{corollary} \label{cor:stable} The two triples
\begin{align*}
(BDM_k(\R^n), \mathcal{P}_{k-1}^d(\R^n), \mathcal{P}_0^d(\R_{\skw}^{n \times n})) \quad \text{and} \quad (RTN_{k+1}(\R^n), \mathcal{P}_k^d(\R^n), \mathcal{P}_0^d(\R_{\skw}^{n \times n}))
\end{align*}
are elasticity stable for $k \geq 1$.
\end{corollary}
Let $b_T$ be the standard cubic/quartic bubble function on a triangle/tetrahedron $T \in \mathcal{T}_h$. In the two and three dimensional cases $B_k \subset \Xi$, $k \geq 1$ is defined by 
\begin{align*}
B_{k} = 
\begin{cases}
\{ \eta \in L^2(\Omega; \R^2) \,|\, \eta|_T \in b_T \mathcal{P}_{k-1}(T; \R^2) \}, &  \text{if } n = 2, \\
\{ \eta \in L^2(\Omega; \R^{3 \times 3}) \,|\, \eta|_T \in b_T \mathcal{P}_{k-1}(T; \R^{3 \times 3}) \}, &  \text{if } n = 3.
\end{cases}
\end{align*}
When $n=3$ a matrix bubble function ${\bf b}_T$ on each $T \in \mathcal{T}_h$ is defined by 
\begin{align*}
{\bf b}_T = \sum_{i=0}^4 \lambda_i \lambda_{i+1} \lambda_{i+2} (\grad \lambda_{i+3})^T (\grad \lambda_{i+3}),
\end{align*}
where $\lambda_i$, $i=0,1,2,3$ are the barycentric coordinates on $T$, $\grad \lambda_i$ is a row vector, and the index $i$ is counted modulo 4. One can see that ${\bf b}_T$ is symmetric positive definite and the cross product of each row of ${\bf b}_T$ and the unit normal vector $n_e$ on an edge/face $e \subset \pd T$ vanishes. By the integration by parts, 
\begin{align} \label{eq:IBP}
(\curl ({\bf b}_T \curl \eta_1), \eta_2) = ({\bf b}_T \curl \eta_1, \curl \eta_2), \qquad \eta_1, \eta_2 \in \mathcal{P}_k(T;\R_{\skw}^{n \times n}),
\end{align}
so the above relation gives an inner product on
\begin{align*}
\hat{\mathcal{P}}_k^d(\R_{\skw}^{n \times n}) = \{ \tau \in \mathcal{P}_k^d(\R_{\skw}^{n \times n}) \,|\, \tau \perp \mathcal{P}_0^d(\R_{\skw}^{n \times n}) \}.
\end{align*}
Moreover, the norm given by this inner product with weight ${\bf b}_T$ is equivalent to the standard $L^2$ norm on $T$ up to constants independent of the diameter of $T$. For $k \geq 1$ let
\begin{align} \label{eq:B-eta}
B(\eta) = 
\begin{cases}
h_T^{-2} b_T \rot \eta & \text{for } \eta \in \hat{\mathcal{P}}_k(T; \R_{\skw}^{2 \times 2})  \quad \text{ if } n = 2, \\
h_T^{-2} {\bf b}_T \curl \eta & \text{for } \eta \in \hat{\mathcal{P}}_k(T; \R_{\skw}^{3 \times 3}) \quad \text{ if }  n = 3,
\end{cases}
\end{align}
and define $\hat{B}_k$ as 
\begin{align*}
\hat{B}_k = \{\xi \in \Xi \,:\, \xi|_T = B(\eta) \quad \text{for some } \eta \in \hat{\mathcal{P}}_k(T; \R_{\skw}^{n \times n}) \}. 
\end{align*}
\begin{lemma} \label{lemma:bub-stable}
For $k \geq 1$ the pairs $(\hat{B}_k, \hat{\mathcal{P}}_k^d(\R_{\skw}^{n \times n}))$ and $(B_k, \hat{\mathcal{P}}_{k}^d(\R_{\skw}^{n \times n}))$ are Stokes stable.
\end{lemma}
\begin{proof}
We prove only the three dimensional case for both pairs because the two dimensional case is similar. 

In the case of $(\hat{B}_k, \hat{\mathcal{P}}_k(\R_{\skw}^{n \times n}))$ pair, for a given $\eta \in \hat{\mathcal{P}}_k(\R_{\skw}^{n \times n})$, take $\xi \in \hat{B}_k$ such that $\xi|_T = h_T^{-2} {\bf b}_T \curl \eta|_T$ for $T \in \mathcal{T}_h$. Then
\begin{align*}
(Q_h \chi \div S \xi, \eta)_T &= (Q_h \skw \curl \xi, \eta)_T = (\curl \xi, \eta)_T = h_T^{-2} ({\bf b}_T \curl \eta, \curl \eta)_T.
\end{align*}
By the standard scaling argument 
\begin{align*}
\| \eta \|_{0,T} \sim h_T \| \curl \eta \|_{0,T} , \qquad 
\| \curl \xi \|_{0,T} \sim h_T^{-1} \| \curl \eta \|_{0,T} \sim \| \eta \|_{0,T},
\end{align*}
so
$\| \curl \xi \|_0 \leq c \| \eta \|_0$ and $(Q_h \chi \div S \xi, \eta) \geq c \| \eta \|_0^2$.

With the $({B}_k, \hat{\mathcal{P}}_k(\R_{\skw}^{n \times n}))$ pair and a given $\eta \in \hat{\mathcal{P}}_k(\R_{\skw}^{n \times n})$, one can take $\xi \in B_k$ such that $\xi|_T = h_T^{-2} b_T \curl \eta|_T$ for $T \in \mathcal{T}_h$. The rest of the argument is then similar to the first case, so we omit the details.
\end{proof}
We are now ready to present examples with discontinuous $\Gamma_h$.

\subsubsection{The Cockburn--Gopalakrishnan--Guzm\'{a}n (CGG) and the Arnold--Falk--Winther (AFW) elements} \label{eg4} The CGG elements \cite{CGG10} are  
\begin{align*}
\Sigma_h = RTN_{k}(\R^n) + \curl \hat{B}_{k-1}, \quad U_h = \mathcal{P}_{k-1}^d (\R^n), \quad \Gamma_h = \mathcal{P}_{k-1}^d(\R_{\skw}^{n \times n}),
\end{align*}
for $k \geq 2$. We apply the framework with $\Gamma_h^0 = \mathcal{P}_0^d(\R_{\skw}^{n \times n})$ and $\Xi_h = \hat{B}_{k-1}$. Then $(\hat{B}_k, \hat{\mathcal{P}}_k^d(\R_{\skw}^{n \times n}))$ is Stokes stable by Lemma \ref{lemma:bub-stable} and $(\Sigma_h, U_h, \Gamma_h^0)$ is elasticity stable by Corollary \ref{cor:stable}. Moreover, $Q_h \chi \div S \xi = Q_h \skw \curl \xi$ for $\xi \in \Xi_h$ is orthogonal to $\Gamma_h^0$ by the definition of $\hat{B}_k$ and \eqref{eq:IBP}, so $(\Sigma_h, U_h, \Gamma_h)$ is elasticity stable. 

The AFW elements \cite{AFW07} are 
\begin{align*}
\Sigma_h = BDM_{k}(\R^n) , \quad U_h = \mathcal{P}_{k-1}^d (\R^n), \quad \Gamma_h = \mathcal{P}_{k-1}^d(\R_{\skw}^{n \times n}), 
\end{align*}
for $k \geq 1$. The stability of these elements for $k = 1$ follows as a corollary of Lemma \ref{lemma:afw1}. For $k \geq 2$ it follows from the stability of CGG elements because $\curl \hat{B}_{k-1} \subset \mathcal{P}_{k}^d(\R^{n \times n}) \cap H(\div; \R^n)$ and then $RTN_{k}(\R^n) + \curl \hat{B}_{k-1} \subset BDM_{k}(\R^n)$.

\subsubsection{The Gopalakrishnan--Guzm\'{a}n (GG) and the Stenberg elements} \label{eg5}
The GG elements \cite{GG10} are 
\begin{align*}
\Sigma_h = BDM_{k}(\R^n) + \curl \hat{B}_{k} , \quad U_h = \mathcal{P}_{k-1}^d (\R^n), \quad \Gamma_h = \mathcal{P}_{k}^d(\R_{\skw}^{n \times n}),
\end{align*}
for $k \geq 1$. Let $\Gamma_h^0 = \mathcal{P}_0^d(\R_{\skw}^{n \times n})$ and $\Xi_h = \hat{B}_{k}$. Then $(\Xi_h, \Gamma_h^1)$ is Stokes stable by Lemma \ref{lemma:bub-stable} and $(\Sigma_h, U_h, \Gamma_h^0)$ is elasticity stable by Corollary \ref{cor:stable}. Moreover, $Q_h \chi \div S \Xi_h$ is orthogonal to $\Gamma_h^0$, so $(\Sigma_h, U_h, \Gamma_h)$ is elasticity stable. By checking the degree of polynomials one can see that $\curl \hat{B}_{k-1} \subset BDM_k(\R^n)$ holds, so only a small part of $\curl \hat{B}_{k}$ is necessary for $\Sigma_h$. See \cite{GG10} for degrees of freedom of $\Sigma_h$ for implementation.

The Stenberg elements \cite{Sten88} are 
\begin{align*}
\Sigma_h = BDM_{k}(\R^n) + \curl {B}_{k} , \quad U_h = \mathcal{P}_{k-1}^d (\R^n), \quad \Gamma_h = \mathcal{P}_{k}^d(\R_{\skw}^{n \times n}),
\end{align*}
for $k \geq 1$. The stability can be proved in a way similar to the GG elements by Lemma \ref{lemma:bub-stable} and Corollary \ref{cor:stable}. As noticed in \cite{Guzman11}, the $k=1$ case, which is not included in \cite{Sten88}, also gives a stable mixed method.

\subsubsection{Elements with barycentric subdivision grids} \label{eg6} Let $\mathcal{M}_h$ be a shape-regular mesh of $\Omega$ and $\mathcal{T}_h$ be the mesh obtained by dividing each element in $\mathcal{M}_h$ into $n+1$ subelements by connecting the vertices of the element to its barycenter. Define
\begin{align*}
\Sigma_h = BDM_{k}(\R^n), \quad U_h = \mathcal{P}_{k-1}^d (\R^n), \quad \Gamma_h = \mathcal{P}_{k}^d(\R_{\skw}^{n \times n}),
\end{align*}
for $k \geq 1$ on $\mathcal{T}_h$. As is noticed in \cite{GG10}, this triple is elasticity stable when $n=2$ and when $n=3$ with $k \geq 2$  due to the stable finite elements for the Stokes equation in \cite{MR2691498,MR2114637}. We will show that it can be extended to $n=3$, $k=1$ case. For $M \in \mathcal{M}_h$ we define  
\begin{align*}
\mathcal{P}_{2,0}^c(M) &= \{ \xi \in H^1(M; \R^{3 \times 3}) \,:\, \xi|_T \in \mathcal{P}_{2}(T; \R^{3 \times 3}), \xi |_{\pd M} = 0 \text{ for }T \subset M, T \in \mathcal{T}_h \},
\end{align*}
and 
\begin{align*}
\Xi_h &= \{ \xi \in \Xi \,:\, \xi|_M \in \mathcal{P}_{2,0}^c(M) \text{ for }M \in \mathcal{M}_h \}, \\
\Gamma_h^0 &= \{ \eta \in \Gamma \,:\, \eta|_M = \mathcal{P}_0^d (M;\R_{\skw}^{n \times n}) \text{ for } M \in \mathcal{M}_h \}.
\end{align*}
Then the restriction of $\Gamma_h^1$ on a macroelement $M$ is the space of mean-value zero polynomials on $M$, so $(\Xi_h, \Gamma_h^1)$ is Stokes stable by Lemma 2 in \cite{MR2114637}. Recall that $\Xi_h$ is a space of continuous piecewise quadratic polynomials on $\mathcal{T}_h$ and one can see that $\curl \Xi_h \subset \Sigma_h$. Furthermore, $\curl \Xi_h$ is orthogonal to $\Gamma_h^0$ due to the integration by parts because $\xi \in \Xi_h$ is a bubble-like function on each macroelement $M \in \mathcal{M}_h$. Note that this is not the case if we simply take $\Gamma_h^0$ as the space of piecewise constants on $\mathcal{T}_h$. Since $(\Sigma_h, U_h, \Gamma_h^0)$ is elasticity stable by Lemma \ref{lemma:afw1}, so is $(\Sigma_h, U_h, \Gamma_h)$. 

A completely analogous argument can be used to show that 
\begin{align*}
\Sigma_h = RTN_{k+1}(\R^n), \quad U_h = \mathcal{P}_{k}^d (\R^n), \quad \Gamma_h = \mathcal{P}_{k}^d(\R_{\skw}^{n \times n}), \qquad k \geq 1,
\end{align*}
is elasticity stable. This readily implies that the finite element family for dual-mixed form of steady Navier--Stokes equations in \cite{MR3056409} can be extended to $k=1$, $n=3$ case.

\subsubsection{Triangular elements with variable degree shape functions} \label{eg6-2} Aiming to $hp$-adaptive methods, Demkowicz and Qiu investigated $h$-stability of the AFW elements with variable degree shape functions. We refer to \cite{MR2557491,MR2792388} for a precise definition of variable degree finite element spaces. Their approach is extending the elasticity complex framework in \cite{AFW07} to variable degree finite element spaces with a suitable interpolation operator satisfying commuting diagram properties for variable degree polynomial spaces. Construction of such an interpolation operator is difficult. In fact, an indirect way to prove its existence in \cite{MR2557491,MR2792388} is nontrivial and the argument seems to be highly sensitive to shape functions of elements. However, although $p$-stability is still missing, an $h$-stability result for variable degree elements (with bounded highest degree) is readily obtained in our framework by taking $\Gamma_h^0$ as the piecewise constant space, $(\Sigma_h, U_h, \Gamma_h^1)$ as suitable variable degree finite element spaces, and by repeating the previous stability proof of the AFW elements. Moreover, the error estimate obtained by \eqref{eq:imp-err} is better than that in \cite{MR2557491,MR2792388}. Lastly, a completely analogous argument will give $h$-stable variable degree elements based on CGG and GG elements.

\subsubsection{Rectangular elements with discontinuous $\Gamma_h$} \label{eg7} There are not much literature on rectangular/quadrilateral weak symmetry elements with discontinuous $\Gamma_h$ \cite{Awanou12,JL1,JL2,AAQ13}. Unfortunately the properties which are crucial in our framework do not hold on quadrilateral meshes. For instance, {\bf (A1)} fails in general for quadrilateral $H(\div)$ and $L^2$ element pairs, and there is no interpolation operator $\Pi_h$ as in Theorem \ref{thm:error-estm}, so we will only consider rectangular meshes. We remark that a family of elements on quadrilateral meshes, say the Arnold--Awanou--Qiu elements, have been developed in \cite{AAQ13}. 

It is known that the triple $(rBDM_1(\R^n), \mathcal{P}_0^d(\R^n), \mathcal{P}_0^d(\R_{\skw}^{n \times n}))$ is elasticity stable \cite{Awanou12}. 
However, these elements are not readily extended to higher orders as in the triangular AFW elements. The higher order elements in \cite{Awanou12} are $(rRTN_{k+2}, \mathcal{Q}_{k+1}^d\mathcal{Q}_{k}^d)$ with $k \geq 1$, which are not rectangular analogues of the AFW elements. The Arnold--Awanou--Qiu family \cite{AAQ13} on rectangular meshes is $(rRTN_{k+1}, \mathcal{Q}_k^d, \mathcal{P}_k^d)$ with $k \geq 1$. As discussed in \cite{AAQ13}, this family can be analyzed with the framework using Stokes stable pairs $(\mathcal{Q}_{k+1}^c, \mathcal{P}_k^d)$, $k \geq 1$ and elasticity stable pairs $(rRTN_{k+1}, \mathcal{Q}_k^d, 0)$, $k \geq 1$. 
Existence of similar higher order elements in three dimensions is not clear.

\begin{table}[t]  
\caption{Convergence rates of errors and orders of approximation of finite elements for some  weak symmetry elements.} \label{mixed-approx}
\centering
\begin{tabular}{>{\small}c >{\small}c >{\small}c >{\small}c >{\small}c >{\small}c}
\hline
\multirow{2}{*}{elements}	&  \multicolumn{3}{>{\small}c}{convergence rate\;[order of approximation]} & \multirow{2}{*}{order } & \multirow{2}{*}{$n$ } \\ 
						&$\| \sigma - \sigma_h \|_0$	& $\| P_h u - u_h \|_0$	&  $\| \gamma - \gamma_h \|_0$		&	\\  \hline \hline
PEERS \eqref{eg1}		& $1\;[1]$		& $1\;[1]$	& $1\;[1]$ 		& $1 $ & $2$ \\ 
THB \eqref{eg2}		& $k\;[k]$		& $k\;[k-1]$	& $k\;[k]$ 		& $k \geq 2$ & $2/3$ \\ 
2D Rect. \eqref{eg3}		& $2\;[2]$		& $2\;[1]$	& $2\;[2]$ 		& $2$ & $2$ \\ 
AFW \eqref{eg4}			& $k\;[k+1]$	& $k\;[k]$	& $k\;[k]$		& $k \geq 1 $ & $2/3$ \\ 
CGG \eqref{eg4}			& $k\;[k]$		& $k\;[k]$	& $k\;[k]$		& $k \geq 2 $ & $2/3$ \\ 
Stenberg \eqref{eg5}	& $k\;[k]$		& $k\;[k-1]$	& $k\;[k]$ 		& $k \geq 2 $ & $2/3$ \\ 
GG \eqref{eg5}			& $k\;[k]$		& $k\;[k-1]$	& $k\;[k]$		& $k \geq 2 $ & $2/3$ \\ 
barycentric BDM \eqref{eg6}	& $k\;[k]$		& $k\;[k-1]$	& $k\;[k]$		& $k \geq 2 $ & $2/3$ \\ 
barycentric RTN \eqref{eg6}	& $k\;[k]$		& $k\;[k]$	& $k\;[k]$		& $k \geq 2 $ & $2/3$ \\ 
Awanou low \eqref{eg7}	& $1\;[2]$		& $1\;[1]$	& $1\;[1]$		& $1 $ & $2/3$ \\ 
Awanou high \eqref{eg7}	& $k\;[k+1]$		& $k\;[k+1]$	& $k\;[k]$		& $k \geq 2 $ & $2$ \\ 
rAAQ \eqref{eg7}	& $k\;[k]$		& $k\;[k]$	& $k\;[k]$		& $k \geq 2 $ & $2$ \\
rGG \eqref{eg7}	& $k\;[k]$		& $k\;[k-1]$	& $k\;[k]$		& $k \geq 2 $ & $2/3$ \\
\hline
\multicolumn{5}{l}{\footnotesize THB = Taylor--Hood based elements, 2D Rect. = 2D rectangular element} \\
\multicolumn{5}{l}{\footnotesize AFW = Arnold--Falk--Winther, CGG = Cockburn--Gopalakrishnan--Guzm\'{a}n } \\
\multicolumn{5}{l}{\footnotesize GG = Gopalakrishnan--Guzm\'{a}n, Barycentric = the elements on barycentric meshes} \\
\multicolumn{5}{l}{\footnotesize rAAQ = Arnold--Awanou--Qiu family on rectangular meshes, rGG = rectangular GG}
\end{tabular}                           
\end{table}

As the last example, we propose a family of new rectangular elements, say, rectangular GG elements. Following the construction of the GG elements we define $\hat{B}^r(\eta)$ for $\eta \in \hat{\mathcal{P}}_k^d(\R_{\skw}^{n \times n})$ as in \eqref{eq:B-eta} but with a standard rectangular quartic bubble function $b_T^r$ in two dimensions, and a rectangular matrix bubble function ${\bf b}_T^r$ in three dimensions. We refer to \cite{JL1} for precise definitions of $b_T^r$, ${\bf b}_T^r$, and a proof of \eqref{eq:IBP} with $b_T^r$ and ${\bf b}_T^r$. The space $\hat{B}_k^r$ is defined correspondingly. The rectangular GG elements are defined by 
\begin{align*}
\Sigma_h = rBDM_{k}(\R^n) + \curl (\hat{B}_{k}^r) , \qquad U_h = \mathcal{P}_{k-1}^d (\R^n), \qquad \Gamma_h = \mathcal{P}_{k}^d(\R_{\skw}^{n \times n}),
\end{align*}
for $k \geq 1$. The stability of these elements can be proved with an analysis similar to the GG elements by taking $\Gamma_h^0 = \mathcal{P}_0^d(\R_{\skw}^{n \times n})$ and $\Xi_h = \hat{B}_k^r$. By counting degrees of polynomials one can see that $\curl (\hat{B}_{k-2}^r)$ is included in $rBDM_{k}(\R^n)$ since $rBDM_{k}$ contains all polynomials of degree $k$ (cf. (3.29) and (3.30) in \cite{BFBook}), so the number of degrees of freedom of $\Sigma_h$ can be reduced because only a part of $\curl (\hat{B}_{k}^r)$ is necessary for $\Sigma_h$.

\section{Conclusion}
We presented a framework for analysis of mixed methods for elasticity with weakly symmetric stress by generalizing the approach in \cite{MR2449101}. The framework enables us to analyze many existing weak symmetry elements in a unified way with elementary techniques. We also showed that some new stable mixed finite elements can be easily obtained from it. 


\bibliographystyle{simunsrtcompress}

\end{document}